\documentclass[12pt]{amsart}

\usepackage{amsmath,amssymb,amsthm}
\usepackage[bookmarksopen=true,
            bookmarksnumbered=true,
            colorlinks=true,
            linkcolor=blue,
            pdftitle={final-submit},
            pdfsubject={Commutative Algebra},
            pdfkeywords={Castelnuovo-Mumford regularity, Koszul cycle, Koszul homology, },
            pdfpagemode=UseOutlines,
            pdfproducer={Latex with hyperref},
            letterpaper=true,
            pdfcreator={latex->dvips->ps2pdf}]{hyperref}
\usepackage{graphicx,epsfig}
\usepackage{enumerate}
\usepackage{xypic}
\usepackage{tikz}
\usepackage{multicol}
\usetikzlibrary{calc}
  \input{xy}
  \xyoption{all}
\usepackage{url}

\numberwithin{equation}{section}

\theoremstyle{plain}
\newtheorem{thm}{Theorem}[section]

\newtheorem{prop}[thm]{Proposition}
\newtheorem{lem}[thm]{Lemma}
\newtheorem{cor}[thm]{Corollary}

\theoremstyle{definition}
\newtheorem{dfn}[thm]{Definition}

\newtheorem{exmp}[thm]{Example}

\newtheorem{rem}[thm]{Remark}

\newtheorem{dfns-rems}[thm]{Definitions and Remarks}
\newtheorem{notas-rems}[thm]{Notations and Remarks}
\newtheorem{exmps-rems}[thm]{Examples and Remarks}

\def\0{{\bf 0}}

\def\NN{\mathbb{N}}

\def\PP{\mathbb{P}}
\DeclareMathOperator{\coker}{coker}

\DeclareMathOperator{\depth}{depth}

\DeclareMathOperator{\reg}{reg}

\DeclareMathOperator{\en}{end}

\DeclareMathOperator{\sym}{Sym}
\DeclareMathOperator{\Tor}{Tor}

\DeclareMathOperator{\inde}{index}

\begin{document}

\author{ kamran lamei and Navid Nemati}

\title{Castelnuovo-Mumford regularity of Koszul cycles and Koszul homologies}

\begin{abstract}

We extend to one dimensional quotients the result of A. Conca and S. Murai on the convexity of the regularity of Koszul cycles. By providing a relation between the regularity of Koszul cycles and Koszul homologies we prove a sharp regularity bound for the Koszul homologies of a homogeneous ideal in a polynomial ring under the same conditions.

\end{abstract}

\maketitle

%%%%%%%%%%%%%%%%%%%%%%%%%%%%%%%%%%%%%%%%%%%%%%%%%%%%%%%%%%%%%%%%%%%%%%%%%%

\section{Introduction} \label{sec1}
In order to attain better understanding of a projective variety classically there is a great interest  to determine the equations of a projective variety and also the syzygies of its homogeneous ideal. In this regard, M. Green and R. Lazarsfeld defined the property $N_p$ which, roughly speaking, refers to the simplicity of syzygies of the homogeneous coordinate ring of a smooth projective variety embedded by a very ample line bundle. More precisely, let $R$ be a finitely generated graded $K$-algebra then we say that $R$ satisfies property $N_p$ if we have $\beta_{i,i+j}(R)=0$ for all $1<j$ and for all $1\leq i\leq p$. The Green-Lazarsfeld index of $I$ denoted by $\inde(I)$ is the maximum of such $p$. In the case $X= \PP^n$ with the line bundle $\mathcal{O}_{\PP^n}(d)$, M. Green proved that the coordinate ring of the image of Veronese embedding of degree $d$ of $X$ satisfies the property $N_d$. W. Bruns A. Conca and T. R\"{o}mer \cite{BCR2} improved the lower bound of the Green-Lazarsfeld index of the Veronese subring $S^{(d)}$ to $d+1$, their approach is based on investigation of the homological invariants of the Koszul cycles and Koszul homologies of $d$-th power of the maximal ideal.

By aforementioned motivation A. Conca and S. Murai studied the Castelnuovo-Mumford regularity of the Koszul cycles $Z_{t}(I,S)$ of a homogeneous ideal in a polynomial ring $S$. They proved that regularity of Koszul cycles $Z_{i} (I,S)$ as a function of $i$ is subadditive when $\dim S/I = 0$. We make a generalization showing that if $S$ is a  polynomial ring over a field of characteristic $0$ and $\dim S/I \leq 1$ then 
$$ \reg (Z_{s+t}(I,S)) \leq \reg (Z_{t}(I,S)) + \reg (Z_{s}(I,S)).$$

From the convexity of the regularity of Koszul cycles in dimension $0$,  A. Conca and  S. Murai \cite[Corollary 3.3]{CM} obtained a bound on the regularity of Koszul homologies. But inspired by the remarkable result of M.Chardin and P. Symonds \cite{CS} on the regularity of  cycles and homologies of  a general complex, first we determine the regularity of Koszul cycles by the regularity of the previous Koszul homologies.
%\begin{theorem} \label{reginduc}
Let $S$ be a  polynomial ring, $I$ be a homogeneous ideal of $S$. If $\dim S/I\leq  1$, then for all $0<i<\mu(I)$
$$
\reg(Z_i(I,S))= \max_{0<j<n} \lbrace \reg(H_{i-j}(I,S))+j+1\rbrace.
$$
Where $\mu(I)$ is the minimal number of generators of $I$.
%\end{theorem}

As an application we state sharp bound for the regularity of Koszul homologies in dimension $1$ which is a refinement of the result of A. Conca and S. Murai in dimension $0$. Let $I$ be an ideal of $S$ and  $\dim S/I \leq  1$, then we have the following inequalities between Koszul homologies of $I$ for all $i,j
\geq 1$
$$
\reg(H_{i+j-1}(I,S))\leq \max_{1\leq \alpha,\beta\leq n-1 }\lbrace \reg(H_{i-\alpha}(I,S))+\reg(H_{j-\beta}(I,S))+\alpha+\beta\rbrace.
$$

In the last section, we investigate the behavior of regularity of Koszul homologies of power of the maximal ideal. Thanks to the equivalence between Betti numbers of Veronese embedding and regularity of Koszul homologies, as an application we give a short proof for the theorem of M. Green in \cite{G} on Green-Lazarsfeld index of Veronese embeddings.
%%%%%%%%%%%%%%%%%%%%%%%%%%%%%%%%%%%%%%%%%%%%%%%%%%%%%%%%%%%%%%%%%%%%%%%%%%
\section{Preliminaries} \label{sec2}
%\subsection*{Minimal free resolution}
Let $S=k[x_1, \dots , x_n]$ be a polynomial ring over a field $k$ and $M$ be a finitely generated graded $S$-module. A minimal free resolution of $M$ is an exact sequence
\begin{center}
$0\rightarrow F_p\rightarrow F_{p-1}\rightarrow \cdots \rightarrow F_0\rightarrow M \rightarrow 0, $
\end{center}
where each $F_i$ is a graded $S$-free module of the form $F_i=\oplus S(-j)^{\beta_{i,j}(M)} $ such that the number of basis elements is minimal and each map is graded. The value ${\beta_{i,j}(M)}$ is called the $i$-th graded Betti numbers of $M$ of degree $j$. Note that the minimal free resolution of $M$ is unique up to isomorphism so the graded Betti numbers are uniquely determined. 
%We refer the reader to ****(Peeva) for further reading.
\begin{dfn}
Let $S$ be a polynomial ring and $I$ be a homogeneous ideal of $S$. We say that $I$ satisfies property $N_p$ if we have $\beta_{i,i+j}(I)=0$ for all $j>1$ and for all $i\leq p$. The Green-Lazarsfeld index of $I$ denoted by $\inde(I)$ is the maximum of such $p$.
\end{dfn}
\begin{rem}
We have $\inde(I)=\infty$ if and only if $I$ is generated in degree $2$ and has a linear resolution.
\end{rem}

Let $I= (f_1,\dots ,f_r)$ be a graded $S$-ideal minimally generated in degrees $d_1,\cdots ,d_r$. Define $K(I,S)= \oplus  K_t(I,S)$ as the Koszul complex associated to the $S$-linear map $\phi : F_0 = \oplus S(- d_i) \rightarrow S$ in which $\phi (e_i)= f_i$. Let $K(I,M) = K(I,S) \otimes M$ and denote $Z_t (I,M)$, $B_t (I,M)$ and $H_t (I,M)$ as the cycles, boundaries and homologies of $K(I,M)$ respectively at the homological position $t$. We use $Z_t(I), B_t(I)$ and $H_t(I)$ whenever $M=S$.

\begin{rem}
The Koszul complex does depend on the choice of the generators, but it is unique up to isomorphism if we choose minimal set of generators. Since we only deal with the case that the set of generators is minimal,  we use $K(I)$ instead of $K(f_1,\dots ,f_r)$.
\end{rem}
%\subsection*{\v{C}ech complex and Castelnuovo-Mumford regularity}
Let $\mathfrak{m}= (x_1\dots , x_n)$ be the graded maximal ideal of $S$, for a finitely generated module $M$ define the \v{C}ech complex as follows:
$$
\mathcal{C}^{\bullet}_{\mathfrak{m}}(M): 0 \rightarrow M \rightarrow \oplus_{1\leq i \leq n} M_{x_i} \rightarrow \oplus_{1\leq i,j\leq n} M_{x_ix_j} \rightarrow \cdots \rightarrow M_{x_1\dots x_n}\rightarrow 0 
$$
%\begin{dfn}(Local Cohomology) Let $S$ be a polynomial ring and $I$ be a graded ideal of $S$. For a finitely generated module $M$ we define local cohomology of $M$ supported in $I$ :
%\begin{center}
%$\textbf{H}^i_I (M):= H_i(\mathcal{C}^{\bullet}_{I}(M)) $
%\end{center}
%\end{dfn}
The local cohomology modules of an S-module $M$ are the homologies of the \v{C}ech complex. It is a well known fact that each local cohomology module is  artinian so we can speak about the last nonzero degree of each of them.  We define $$a_i^{\mathfrak{m}}(M):= \en (H^i_{\mathfrak{m}}(M)) =\max\lbrace j\,\vert \,H^i_{\mathfrak{m}}(M)_j\neq 0\rbrace,$$ then the Castelnuovo-Mumford regularity of module $M$ is defined as follows:
\begin{dfn} 
Let $S=k[x_1,\dots ,x_n]$ be a polynomial ring,  $\mathfrak{m}=(x_1,\dots , x_n)$ be a unique graded maximal ideal of $S$ and $M$ be a finitely generated $S$-module, then
\begin{equation*}
\reg(M)= \max \lbrace a_i^{\mathfrak{m}}(M)+i\rbrace 
\end{equation*}
\end{dfn}
\subsection*{Veronese embedding}
Let $S=k[x_1,\dots , x_n]$ be a polynomial ring, for each $c$ consider $S$-subalgebra $V_S(c)= \oplus_{i\geq 0} S_{ic}$ where $S_d$ is a $k$-vector space spanned by all monomials of degree $c$ in $S$. We call $V_S(c)$ the $c$-th Veronese embedding of $S$. There is another interpretation of Veronese embedding as a coordinate ring of a variety, which is isomorphic to $k[t_1,\dots ,t_N]/J$ where $N= \binom{n+c-1}{n-1}$ and $J$ is a kernel of the map $\phi: k[t_1,\dots , t_N] \rightarrow S$ by sending $t_i$'s to all monomials of degree $c$ in $S$. The variety correspond to $J$ is called Veronese Variety. We denote $\beta_{i,j}(k[t_1,\dots ,t_N]/J)$ by $\beta_{i,j}(V_S(c))$ 
\section{Regularity of Koszul cycles} \label{sec2}
In this section we will present a generalization of a result about convexity of regularity of Koszul cycles of A. Conca and S. Murai.

 The following interesting property of Koszul cycles was first investigated in work of W. Bruns, A. Conca and T. R\"{o}mer in \cite{BCR}
\begin{lem}\cite[Lemma 2.4]{BCR}  \label{summand}
Let $S$ be a polynomial ring, $I$ be a homogeneous ideal of $S$ and $M$ be a finitely generated graded $S$-module. Suppose that the element $\left(\begin{array}{ccc} s+t \\ s\end{array}\right)$ is invertible in $S$. Then $Z_{s+t} (I,M)$ is a direct summand of $Z_{s} (I,Z_{t}(I,M))$
\end{lem}

The following lemma provides us to compare regularities of diferrent terms of exact sequences and basically it plays the main role in the generalization of the result of A. Conca and S. Murai\cite{CM} on the convexity of regularity of  Koszul cycles. 
\begin{lem} \label{4term}
Let  $L$ : $0 \longrightarrow{} L_{4}  \overset{{d_4}}\longrightarrow   L_{3}\overset{{d_3}} \longrightarrow L_{2} \overset{{d_2}}\longrightarrow L_{1} \longrightarrow 0$ be an exact sequence of finitely generated graded $S$-modules such that $L_1$ and $L_4$ are supported in dimension less than or equal $1$, and $\depth L_2\geq 2$ then
%$$
%(\ast )\hspace{5mm} \dim L_{1} , L_{4} \leqslant 2\hspace{3mm} and \hspace{3mm}   %%\depth L_{2}  \geqslant 2. 
% $$\\
% Then \\\\
$$\reg (L_{3}) = \max \{\reg (L_{4}) , \reg (L_{2} ), \reg (L_{1}) - 1 \},$$
in particular $\reg (L_{2}) \leq \reg (L_{3})$.
\end{lem}
%\hspace*{-1.4cm} (\I\hspace*{-0.1cm}\I) \hspace*{1.5cm}
\begin{proof}
First we decompose the complex $L$ into the following short exact sequences
$$
0 \longrightarrow{} L_{4}  \overset{{d_4}}\longrightarrow   L_{3}\overset{{can.}} \longrightarrow \coker ({d_4}) \longrightarrow 0 ,$$

$$ 0\longrightarrow{}  \coker ({d_4})  \overset{{\bar{d_3}}}\longrightarrow   L_{2}\overset{{d_2}} \longrightarrow L_{1} \overset{{}}\longrightarrow 0.$$
%Let $C$ be the \v{C}ech complex on a homogeneous system of parameters of $R$. Consider the double complex $X = C \otimes L$ where $X_{p,q} = C^{-p} \otimes_{R} L_{q}$, and it give rise to spectral sequences which is at the firs level :\\\\
%\vspace{3mm}
%$$
%\begin{matrix}
%0 &  H^{0}_{m}(L_4)   &  H^{0}_{m}(L_3) & 0 &  H^{0}_{m} (L_1)  & 0 \\\\
%0 &  H^{1}_{m}(L_4)   &  H^{1}_{m}(L_3) & 0 &  H^{1}_{m} (L_1)  & 0 \\\\
%0 &  0   &  H^{2}_{m}(L_3) & H^{2}_{m} (L_2)  & 0  & 0 \\\\  
%0 &  0   &  H^{3}_{m}(L_3) & H^{3}_{m} (L_2)  & 0  & 0      
%\end{matrix}
%$$
Given the above short exact sequences, one can obtain the following induced long exact sequences on local cohomology: 
$$
\hspace*{-1cm}(\text{I})\hspace*{1.5cm}\cdots\longrightarrow H^{i}
_{\mathfrak{m}}(L_{4})\rightarrow H^{i}_{\mathfrak{m}}(L_{3}) \rightarrow H^{i}_{\mathfrak{m}}(\coker 
({d_4})) \rightarrow  H^{i+1}_{\mathfrak{m}}(L_{4})\longrightarrow \cdots
$$
$$
\hspace*{-0.6cm}(\text{II})\hspace*{1.5cm}\cdots \longrightarrow H^{i}_{\mathfrak{m}}(L_{2}) \rightarrow H^{i}_{\mathfrak{m}}(L_1) \rightarrow H^{i+1}_{\mathfrak{m}}(\coker ({d_4})) \rightarrow H^{i+1}_{\mathfrak{m}}(L_{2}) \longrightarrow\cdots
$$

$H^i_{\mathfrak{m}}(L_4)=0$ for $i\geq 2$ as $\dim L_4\leq 1$ thus $(\text{I})$ gives 
\begin{equation}\label{equation1}
 H^{2}_{\mathfrak{m}}(L_3) \cong H^{2}_{\mathfrak{m}}(\coker ({d_4})).
\end{equation}
As $\dim L_1 \leq 1$, by $(\text{II})$ we have 
\begin{equation*}
H^{i}_{\mathfrak{m}}(\coker ({d_4}))\cong H^{i}_{\mathfrak{m}}(L_2) \,\, \forall i\geq 3. 
\end{equation*}
As $\depth L_2 \geq 2$  and $H^i_{\mathfrak{m}}(L_1)=0$ for $i=0,1$, by $(\text{II})$
\begin{equation}\label{equation3}
H^{0}_{\mathfrak{m}}(L_1) \cong H^{1}_{\mathfrak{m}}(\coker ({d_4}))\,\, \text{and} \,\, H^0_{\mathfrak{m}}(\coker(d_4))=0.
\end{equation}
From the exact sequences $(\text{I})$ and $(\ref{equation3})$ we get the following short exact sequences
%\vspace{3mm}
$$0 \rightarrow H^{1}_{\mathfrak{m}}(L_4) \rightarrow H^{1}_{\mathfrak{m}}(L_3) \rightarrow H^{0}_{\mathfrak{m}} (L_1) \rightarrow 0,$$
also  the exact sequences $(\text{II})$ and $(\ref{equation1})$ give
$$0 \rightarrow H^{1}_{\mathfrak{m}}(L_1) \rightarrow H^{2}_{\mathfrak{m}}(L_3) \rightarrow H^{2}_{\mathfrak{m}} (L_2) \rightarrow 0.$$
As a result we have
$$
a^i_{\mathfrak{m}}(L_3)=  \left\{ \begin{array}{ll}
          a^0_{\mathfrak{m}}(L_4)& \mbox{if $i=0 $}\\
      \max \lbrace a^1_{\mathfrak{m}}(L_4), a^0_{\mathfrak{m}}(L_1) \rbrace & \mbox{if $i=1 $}\\
      \max \lbrace a^2_{\mathfrak{m}}(L_2), a^1_{\mathfrak{m}}(L_1) \rbrace & \mbox{if $i=2 $}\\
        a^i_{\mathfrak{m}}(L_2)& \mbox {if $i\geq 3$} 
       \end{array} \right.
$$
which proves the statement.
\end{proof}
\begin{prop}\label{regZ_st}
Let $S=k[x_1,\dots ,x_n]$ be a polynomial ring, let $M$ be a finitely generated graded $S$-module with $\depth M\geq 2$ and let $I$ be a graded ideal of $S$ such that $\dim  S/I \leqslant 1$, then 
%Assume that characteristic of $k$ is $0$ or bigger than $s+t$. Then
$$
\reg (Z_t (Z_{s} (I , M))) \leqslant \reg( Z_t(I)) +  \reg (Z_s (I,M)).
$$
%$$
%\reg (Z_{s+t} (I,M) )\leqslant \reg (Z_t(I,M)) + \reg (Z_{s} (I , M))
%$$
\end{prop}
\begin{proof}
By definition one has the following exact sequences,
\begin{equation*}
(\dagger)\hspace*{1.5cm}0 \rightarrow Z_{s} (I , M)) \rightarrow K_s(I, M) \overset{{d_t}}\rightarrow K_{s-1}(I,M)
\end{equation*}
\begin{equation*}
(\ddagger)\hspace*{1.5cm}0 \rightarrow Z_t (I,Z_{s} (I , M)) \rightarrow K_t(I, Z_s(I,M)) \overset{{d_t}}\rightarrow K_{t-1}(I,Z_s(I,M))
\end{equation*}
Note that  $K_s(I, M)$  and  $K_{s-1}(I, M)$ are  direct sums of copies of $M$,   $(\dagger)$ then implies that $\depth Z_{s} (I , M) \geqslant \min \{2,\depth M\}=2$.  Using $(\ddagger)$, $\depth Z_t (I,Z_{s} (I , M)) \geq \min \lbrace 2, \depth Z_s(I,M)\rbrace=2 $. 
 
 For the canonical map in  \cite[section 5] {BCR}
$$u_{s,t} : Z_t(I) \otimes Z_{s} (I , M) \rightarrow Z_t (I,Z_{s} (I , M))$$
Proposition 5.1 in \cite{BCR} gives an exact sequence,\\
$$
0 \rightarrow \ker ( u_{s,t})\rightarrow Z_t(I) \otimes Z_{s} (I , M) \rightarrow Z_t (Z_{s} (I , M)) \rightarrow \Tor_{1}^{S} (\dfrac{K_{s-1}(I,M)}{B_{s-1}(I,M) } ,Z_t(I))  \rightarrow 0.
$$ \\
Notice that after localization at prime ideals not in the support of $S/I$ all the Koszul cycles become direct sum of copies of $M$ and the map $u_{s,t}$ becomes an isomorphism. Therefore  $\Tor_{1}^{S} (\dfrac{K_{s-1}(I,M)}{B_{s-1}(I,M) } ,Z_t(I))$ and $\ker (u_{s,t})$ are supported in $S/I$, hence have a dimension at most $1$.

Thus the conditions of Lemma \ref{4term} are fulfilled, and this lemma gives:
 $$\reg (Z_t (Z_{s} (I , M))) \leqslant \reg (Z_t(I,M) \otimes Z_{s} (I , M)).$$ Notice that  $\Tor_{1}^{R} (Z_t(I) , Z_s(I,M))$ has Krull dimension at most $1$ because $Z_t(I)$ is free when we localize at prime ideals not in the support of $S/I$, so we apply Corollary 3.1 in \cite{EHU} to get 
 $$
  \reg (Z_t(I,M) \otimes Z_{s} (I , M) )\leqslant  \reg (Z_t(I) )+  \reg (Z_s (I,M)).
 $$
% To finished the proof we use the Lemma \ref{summand} which gives us the following in equality\\
 %$$
 %\reg (Z_{s+t} (I, M) )\leqslant \reg (Z_t (Z_{s} (I , M))).
 %$$
 As a result we get 
 $$
\reg (Z_t (Z_{s} (I , M))) \leqslant \reg( Z_t(I) )+  \reg (Z_s (I,M)).
$$
\end{proof}
\begin{thm}\label{convex}
Let $S=k[x_1, \dots , x_n]$ and $I$ be a graded ideal of $S$, if $\dim S/I \leqslant 1$ and characteristic of $k$ is $0$ or bigger than $s+t$, then
$$
 \reg (Z_{s+t}(I)) \leqslant  \reg( Z_t(I) )+  \reg (Z_s(I)).
$$
\end{thm}
\begin{proof}
The theorem follows from Proposition \ref{regZ_st} and Lemma \ref{summand}.
\end{proof}
\section{regularity of Koszul homologies } \label{sec3}
We start this section by a fact which is likely part of folklore but we did not find in the classical references.
\begin{prop}\label{zi in mKi}
Let $S=k[x_1,\dots , x_n]$ be a polynomial ring and $I$ be an ideal of $S$ minimaly generated by $f_1\dots, f_r$, then $Z_i(I)\subset \mathfrak{m}K_i(I)$ for all $i$.
% In particular, $\reg(Z_i(I))> d_1+\cdots +d_i$.
\end{prop}
\begin{proof}
Suppose it is not, then there exist $z\in Z_i(I)$ that is not in $\mathfrak{m}K_i(I)$. By symmetry we may assume it has the form:
$$
z= e_1\wedge\cdots \wedge e_i + \sum_{j>i}c_j e_1\wedge \cdots\wedge e_{i-1}\wedge e_j+ \text{terms without }e_1\wedge \cdots \wedge e_{i-1}.
$$
Since $\partial(z)=0$ it follows that $(-1)^if_i+ \sum_{j>i}(-1)^jc_jf_j=0 $, as it is the coefficient of $e_1\wedge \cdots \wedge e_{i-1}$ in the expression of $\partial(z)$, which is a contradiction with the fact that $f_1,\dots , f_r $ is a minimal set of generators for $I$.
\end{proof}
%Definition of Veronese embedding or Veronese subalgebra.\\
%In the rest of the section we consider $I=(f_1,\dots , f_r)$  an ideal of $S$ such that $\lbrace f_1,\dots , f_r\rbrace$ is a  minimal generating set of $I$ and $\deg(f_i)=d_i$ where $d_1\geq d_2\geq \cdots \geq d_r$.
\begin{cor}\label{regZ_igeqd_1+...d_i}
Let $S=k[x_1,\dots , x_n]$ be a polynomial ring and $I=(f_1,\dots , f_r)$ be a homogeneous ideal of $S$. Let $ f_1,\dots , f_r$ be a  minimal generating set of $I$ and $\deg(f_i)=d_i$ where $d_1\geq d_2\geq \cdots \geq d_r$, then  $\reg(Z_i(I))> d_1+\cdots +d_i$.
\end{cor}
\begin{proof}
Fix a basis element $e_{1}\wedge\cdots \wedge e_{i}\in K_i$. Since $K_{\bullet}(I)$ is a complex, 
$$
\partial(e_{1}\wedge\cdots \wedge e_{i}\wedge e_{i+1})= (-1)^{i+1}f_{i+1} e_1\wedge \cdots \wedge e_i+ \sum_{0<j<i+1}(-1)^j f_j e_1\wedge \cdots \wedge \hat{e_j}\wedge \cdots \wedge e_{i+1} \in Z_i(I).
$$
Therefore an element of the form  $ge_{1}\wedge\cdots \wedge e_{i}$ should appear as a summand in a minimal generating element of $Z_i(I)$. By Proposition \ref{zi in mKi}, $g\in \mathfrak{m}$. So there exist minimal generator of degree at least $d_1+\cdots + d_i+1$. Hence $\reg(Z_i(I))> d_1+\cdots +d_i$.
\end{proof}
M. Chardin and P. Symonds state an approach in \cite{CS} for investigating the regularity of cycles of a general complex by the regularity of previous homologies. Here we determine a concrete relaion between regularity of cycles and homologies of a koszul complex.
% will provide a semi- additive inequalities for the regularities of Koszul homologies. First of all we mention the main theorem in this section which provide a relation between regularity of Koszul cycles and Koszul homologies.
\begin{thm} \label{reginduc}
Let $S=k[x_1,\dots , x_n]$ be a polynomial ring and $I$ be a homogeneous ideal of $S$ minimally generated by $f_1,\dots ,f_r$. If $\dim S/I\leq  1$, then for $0<i<r$:
\begin{equation}
\reg(Z_i(I))= \max_{0<j<n} \lbrace \reg(H_{i-j}(I))+j+1\rbrace.
\end{equation}
\end{thm}
\begin{proof}
Let $I=(f_1,\dots , f_r)$ and $\deg(f_i)=d_i$ where $d_1\geq d_2\geq \cdots \geq d_r$. 
Let $K_{\bullet}^i(I)$ be the $i$-th truncated Koszul complex of $I$ as follows:
\begin{center}
 $K_{\bullet}^i(I) : 0\rightarrow Z_i(I) \longrightarrow^{^{\hspace{-0.5cm} {{\partial}^{\prime}_i}} } {\hspace{0.2cm}}K_i(I) \longrightarrow ^{^{\hspace{-0.5cm} {{\partial}_i} }} {\hspace{0.2cm}} K_{i-1}(I)\longrightarrow ^{^{\hspace{-0.6cm} {{\partial}_{i-1}} }} {\hspace{0.2cm}} \cdots \longrightarrow ^{^{\hspace{-0.5cm} {{\partial}_1} }} {\hspace{0.2cm}} K_0(I) \rightarrow 0$
 \end{center} 
and $C^{\bullet}$ be the \v{C}ech complex. Consider double complex $X= C^{\bullet}\bigotimes K_{\bullet}^r(I)$ 
 where $X_{p,q}= {C^{\bullet}}^{-p}\otimes {K_{\bullet}^r}(I)_{q}  $, and its associated spectral sequence. We first compute homology vertically and we get 
$$\begin{matrix}
      H^{0}_{\mathfrak{m}} (Z_i(I)) & 0 & 0 & \cdots & 0    \\\\
      H^{1}_{\mathfrak{m}} (Z_i(I)) & 0 & 0 & \cdots & 0    \\\\
      H^{2}_{\mathfrak{m}} (Z_i(I))  & 0 & 0 & \cdots & 0    \\\\
      \vdots & \vdots & \vdots & \vdots &\vdots \\\\
      H^{n}_{\mathfrak{m}} (Z_i(I))  \longrightarrow ^{^{\hspace{-0.6cm} {{H^n_{\mathfrak{m}}(\partial^{\prime}_i) }}}}   &H^{n}_{\mathfrak{m}} (K_i(I)) \longrightarrow ^{^{\hspace{-0.7cm} {{H_{\mathfrak{m}}^n(\partial}_i)} }} &H^{n}_{\mathfrak{m}} (K_{i-1}(I))  \longrightarrow ^{^{\hspace{-0.7cm} {{H_{\mathfrak{m}}^n(\partial}_{i-1})} }} & \cdots \longrightarrow ^{^{\hspace{-0.5cm} {{H_{\mathfrak{m}}^n(\partial}_1)} }} & H^{n}_{\mathfrak{m}} (K_0(I)).
\end{matrix}
$$
By continuing the process we have: 
\[ E^{\infty}_{p,q}= E^2_{p,q} \left\{ \begin{array}{ll}
         H^p_{\mathfrak{m}}(Z_q(I)) & \mbox{if $q=i+1, p<n $}\\
       H_q(H^n_{\mathfrak{m}}(K_{\bullet}^i(I))) & \mbox{if $p=n, q\leq i$}\\
       \ker (H_{\mathfrak{m}}^n(\partial'_i))& \mbox {if $(p,q)=(n,i+1)$} \\
       0 & \mbox{Otherwise}
       .\end{array} \right. \]

Notice that since $a^n_{\mathfrak{m}}(K_j(I))= d_1+\cdots + d_j-n$, it follows that for all $0\leq q\leq i$ we have $\en (E^{\infty}_{n,q})\leq \en (E^1_{n,q})= d_1+\cdots + d_q-n$. 
%Furthermore, $\en ( E^1_{n,i})= \en (E^{\infty}_{n,i})$ as $\en (E^{1}_{n,i}) > \en (E^1_{n,i-1})$ and $\en (E^1_{n,i+1})=0$.

On the other hand, if we start taking homology horizontally we have ${E^{\prime}}^2_{p,q}= H^p_{\mathfrak{m}}(H_q(I)) $ for all $p$ and $q < i$ and ${E^{\prime}}^2_{p,q}=0$ for $q=i,i+1$. Notice that $\dim H_i(I)\leq \dim S/I\leq 1$, therefore spectral sequence collapses in the second page and we have:

\[ {E^{\prime}}^{\infty}_{p,q}= {E^{\prime}}^2_{p,q} \left\{ \begin{array}{ll}
         H^p_{\mathfrak{m}}(H_q(I))& \mbox{if $p= 0,1 \,\, \text{and}\,\, q<  i  $}\\
       0 & \mbox {otherwise} 
       .\end{array} \right. \]

The comparison of two spectral sequences gives
\begin{align*}
 &H^0_{\mathfrak{m}}(Z_i(I))= H^1_{\mathfrak{m}}(Z_i(I))=0\\
 &a^2_{\mathfrak{m}}(Z_i(I))= a^0_{\mathfrak{m}}(H_{i-1}(I))\\
 & a^j_{\mathfrak{m}}(Z_i(I))= \max \lbrace a^1_{\mathfrak{m}}(H_{i-j+2}(I)), 
 a^0_{\mathfrak{m}}(H_{i-j+1}(I))\rbrace,\,\,  \forall\,\, 2<j<n
  \end{align*}
  In addition, for the last local cohomology we have 
 $$a^n_{\mathfrak{m}}(Z_i(I))\leq \max \lbrace a^1_{\mathfrak{m}}(H_{i-n+2}(I)), 
 a^0_{\mathfrak{m}}(H_{i-n+1}(I)), d_1+\cdots d_i-n \rbrace,$$
 furthermore 
 $$
 a^n_{\mathfrak{m}}(Z_i(I))=  \max \lbrace a^1_{\mathfrak{m}}(H_{i-n+2}(I)), 
 a^0_{\mathfrak{m}}(H_{i-n+1}(I))\rbrace
 $$
if $a^n_{\mathfrak{m}}(Z_i(I)) > d_1+\cdots d_i-n$.  By corollary \ref{regZ_igeqd_1+...d_i},  we can deduce that

$$a^n_{\mathfrak{m}}(Z_i(I))=  \max \lbrace a^1_{\mathfrak{m}}(H_{i-n+2}(I)), 
 a^0_{\mathfrak{m}}(H_{i-n+1}(I))\rbrace \,\, \text{or} \,\,a^n_{\mathfrak{m}}(Z_i(I))+n < \reg (Z_i(I))$$
 
In addition, the comparison of the two spectral sequences and Corollary  \ref{regZ_igeqd_1+...d_i} give
 $$a^1_{\mathfrak{m}}(H_{i-n+1}(I))\leq \en  (E^{\infty}_{n,i-1})\leq  d_1+\cdots + d_{i-1}-n< \reg(Z_i(I))-n .$$
 As a result we have:
 \begin{align*}
\reg(Z_i(I)) &=  \max_{0\leq j\leq n} \lbrace a^j_{\mathfrak{m}}(Z_i(I)))+j\rbrace\\
&= \max_{3\leq j\leq  n} \lbrace a^0_{\mathfrak{m}}(H_{i-1}(I))+2, a^1_{\mathfrak{m}}(H_{i-j+2}(I))+j, a^0_{\mathfrak{m}}(H_{i-j+1}(I)) +j ,d_1+\cdots d_i \rbrace\\
&=\max_{2\leq j\leq n} \lbrace \reg(H_{i-j+1}(I))+j\rbrace.
\end{align*}
 
%$$\max_{j<n} \lbrace a_j^{\mathfrak{m}}(Z_i(I)))+j\rbrace= \max_{j<n} \lbrace \reg(H_{i-j+1}(I)+j)\rbrace.$$
%f $\reg(Z_i(I))> a_n^{\mathfrak{m}}(Z_i(I))+n$ then we are done. So suppose that we have an equality.
%Now we only need to control $a_n^{\mathfrak{m}}(Z_i(I))$. 
%By Lemma \ref{lem-Zi} we know that 
%$$a_n^{\mathfrak{m}}(Z_i(I))\geq \reg(H_{i-n+1}(I)),$$
% if equality holds then we are done, so for ending the proof assume $%$a_n^{\mathfrak{m}}
% (Z_i(I))> \reg(H_{i-n+1}(I))$$. 
% By Proposition \ref{regZless than d_1+d_i} $a_n^{\mathfrak{m}}(Z_i(I))\leq d_1+\cdots +d_i -n $ and by Lemma \ref{zi in mKi}, $$
% \reg(H_{i-n+1}(I)) +n< a_n^{\mathfrak{m}}(Z_i(I))+n < \reg(Z_i(I))$$ which %shows 

%\begin{align*}
%\reg(Z_i(I)) &=  \max_{j<n} \lbrace a_j^{\mathfrak{m}}(Z_i(I)))+j\rbrace\\
%&= \max_{0\leq j\leq n} \lbrace a_j^{\mathfrak{m}}(Z_i(I)))+j\rbrace\\%
%&= \max_{0\leq j< n} \lbrace \reg(H_{i-j+1}(I))+j\rbrace\\
%&=\max_{0\leq j\leq n} \lbrace \reg(H_{i-j+1}(I))+j\rbrace.
%\end{align*}

%a_t(Z_i(I))= \max_{t\geq 0} \lbrace a_0(H_{i-t+1}), a_1(H_{i-t+2})\rbrace
%$$
%the assertion follows from 
%$$
%reg(Z_i)= \max_{t\geq 1}\lbrace reg(H_{i-t})+t+1 \rbrace
%$$
\end{proof}
\begin{rem}
From the proof of the Theorem \ref{reginduc}, the following equality also holds
$$
\reg(Z_i(I))= \max_{j>0} \lbrace \reg(H_{i-j}(I))+j+1\rbrace.
$$
\end{rem}
% Now we are ready to state a semi-additivity relation between Koszul homologies.
As a consequence of the Theorems \ref{reginduc} and \ref{convex} we give a regularity bound for Koszul homologies in dimension at most $1$.

%As an application of the Theorem \ref{reginduc} we present a refinement of the result of A. Conca and S. Murai \cite{CM} on semi-convexity of the regularity of Koszul homologies. 
 \begin{thm} \label{convexityhomology}
 Let $S=k[x_1,\dots , x_n]$ be a polynomial ring and $I$ be a homogeneous ideal of $S$. If $\dim S/I\leq  1$, then for all $i,j\geq 1$ we have the following regularity bound  for the  Koszul homologies of $I$.\\
\begin{equation}\label{equation}
\reg(H_{i+j-1}(I))\leq \max_{0< \alpha,\beta<n}\lbrace \reg(H_{i-\alpha}(I))+\reg(H_{j-\beta}(I))+\alpha+\beta\rbrace.
\end{equation}
\end{thm}
\begin{proof}
By Theorem \ref{convex}  we have the following inequality for all $i,j$

$$\reg(Z_{i+j}(I)) \leq \reg(Z_i(I))+\reg(Z_j(I)).$$

By using Theorem \ref{reginduc} we have

\begin{align*}
\reg(H_{i+j-1}(I))+2 &\leq   \reg(Z_{i+j}(I))\\
&\leq \reg(Z_i(I))+\reg(Z_j(I))\\
&= \max_{0< \alpha< n}\lbrace \reg(H_{i-\alpha}(I))+\alpha+1\rbrace +\max_{0<\beta < n} \lbrace \reg(H_{j-\beta}(I))+\beta+1\rbrace\\
&= \max_{0<\alpha,\beta<n}\lbrace \reg(H_{i-\alpha}(I))+\reg(H_{j-\beta}(I))+\alpha+\beta +2\rbrace.
\end{align*}
\end{proof}
%our bound in the **** refined the bound announced in Conca and Murai in *****.  
The following example shows the deviation degree of our bound comparing  to the bound provided by A. Conca and S. Murai in dimension $0$. 
\begin{exmp} 
Let $S=k[x,y,z]$ be a polynomial ring and $I = (x,y,z)^4$. We compare our bound for the regularity of $H_{12}(I)$ for different $i,j$ by the bound in the \cite{CM}. By using MACAULAY2 \cite{M2}  one can see that the $\reg(H_{12}(I))=57$. For bounding regularity of  $H_{12}(I)$ we should choose $i,j$ such that $i+j=13$. By choosing $(i,j)=(1,12)$ (respectively $(2,11), (3,10), (4,9), (5,8), (6,7)$) the right hand side of \ref{equation} is $57$ (respectively $58, 58, 59, 59, 58$). 
On the other hand in the bound proposed by A. Conca and S. Murai the best possible estimate is $61$.
%On the other hand by choosing $(i,j)=(6,6)$ (respectively $(5,7), (4,8), (3,9),  (2,10), (1,11)$)  we get $62$ (respectively, $62, 62, 63, 62, 61, 61$) in the bound proposed by  A. Conca and S. Murai.
\end{exmp}

\begin{cor}
Let $S=k[x_1,\dots , x_n]$ be a polynomial ideal and $I$ be an ideal of $S$. If $\dim S/I\leq 1$, then
\begin{center}
$\reg(H_c(I))\leq (c+1)\reg(H_0(I))+2c $.
\end{center}
\end{cor}
 \begin{proof}
We prove by induction. For $c=1$ by \ref{convexityhomology} we have
$$
\reg(H_1(I))\leq \lbrace \reg(H_0(I))+1+\reg(H_0(I))+1\rbrace = 2\reg(H_0(I))+2 .
$$
Let $\reg(H_i(I))\leq (i+1)\reg(H_0(I))+2i $ for all $i\leq r$, by choosing $i=1$ and $j=r+1$ in \ref{equation} we have
$$
\reg(H_{r+1}(I))\leq \max_{0<\beta<n} \lbrace \reg(H_0(I))+\reg(H_{r+1-\beta}(I))+\beta+1\rbrace
$$
For  all $0< \beta<n$ we have
\begin{align*}
reg(H_{r+1-\beta}(I))+\beta+1&\leq (r-\beta+2)\reg(H_0(I))+2(r+1-\beta)+\beta+1\\
&\leq (r+1)\reg(H_0(I))+2(r+1).
\end{align*}
 
Therefore $\reg(H_{r+1}(I))\leq (r+2)\reg(H_0(I))+2(r+1).$
 \end{proof}
\section{Green-Lazarsfeld index of veronese embedding}
Let $X$ be a smooth projective variety with a very ample line bundle $L$ which sets up an embedding into projective space $\PP^r$ where $r= h^{0} (X,L) -1$. Let $S= \sym H^{0} (X,L)$ be the homogeneous coordinate ring of $\PP^r$ and if we define $R: = \oplus H^{0}(X, \mathcal{O}(kL))$, then $R$ can be viewed as a finitely generated graded $S$-module. The syzygies of $R$ as an $S$-module is investigated by M. Green. Let $X$ be a curve of genus $g$ and let $\mathcal{L}$ be a very ample line bundle on $X$  M. Green proved that if $\deg \mathcal{L} = 2g+p+1$ then the embedding defined by $\mathcal{L}$ has property $N_p$.
In the case of Veronese embedding of projective spaces $\mathcal{\varphi}_c : \PP^{n-1} \rightarrow \PP^{N}$ M. Green proved that the Veronese subring $S^{(c)} = \bigoplus_{i\in \NN}  S_{ic}$ satisfies property $N_c$ and then W. Bruns A. Conca and T. R\"{o}mer extend the lower bound to $c+1$.

G. Ottaviani and R. Paolletti  \cite{OP} proved that the Veronese embedding $\mathcal{\varphi}$ for $n\geqslant 2$ and $c\geqslant 3$ does not satisfy property $N_{3c-2}$. In zero characteristic therefore one can deduce that $$c+1 \leqslant \inde (S^{(c)}) \leqslant 3c-3.$$
G. Ottaviani and R. Paolletti showed that if $n=3$ then $\inde (S^{(c)}) =3c-3$ and they conjectured  that the equality holds for arbitrary $n\geq 3$. Recently, T. Vu proved the conjecture of G. Ottaviani and R. Paoletti in the case $c=4$ \cite{V}.
The following theorem provides an equivalence between the study of the syzygies of Veronese embedding and the study of Koszul homologies of powers of the maximal ideal.
\begin{thm}\label{bettiveronese}
Let $S=k[x_1,\dots ,x_n]$. For $i\in \mathbb{N}$ and $j\in \mathbb{Z}$ we have:
\begin{equation}
\beta_{i,j}(V_S(c))= \dim_K H_i(\mathfrak{m}^c, R)_{jc}
\end{equation}
\end{thm}
\begin{proof}
See 4.1 in \cite{BCR} for the proof.
\end{proof}
Now we use our results on regularity of Koszul cycles and Koszul homologies to find a lower bound for the Green-Lazarsfeld index of Veronese embedding. In this regard, we are able to  reproof the statement of M. Green in \cite{G}.
\begin{lem} \label{increasehomology}
Let $S=k[x_1,\dots,x_n]$ and $V_s(c)$ be the $c$-th Veronese subring of $S$, then
$$
\reg(Z_{i+2}(\mathfrak{m}^c))= \reg(H_{i+1}(\mathfrak{m}^c))+2, \,\, \text{for}\,\,  i\leq \min\lbrace \inde(V_S(c)), 2c\rbrace
$$
\end{lem}
\begin{proof}
By Theorem \ref{reginduc}, it suffices to prove that the regularity of Koszul homologies are increasing as a function of $i$. Since $i\leq \inde(V_S(c))$, by Theorem \ref{bettiveronese}
$$H_i(\mathfrak{m}^c)_{(i+2)c}=0\,\, \text{and}\,\, \reg(H_{i+1}(\mathfrak{m}^c))\geq (i+2)c.$$
As $H_i(\mathfrak{m}^c)_{(i+2)c}=0$, $\reg(H_i(\mathfrak{m}^c))< (i+2)c$ if and only if $H_i(\mathfrak{m}^c)$ has no generator of degree greater than $(i+2)c$. Hence it suffices to show that $\reg(Z_i(\mathfrak{m}^c))\leq (i+2)c$. 
By \ref{convex} and \ref{reginduc} we have
$$
\reg( Z_i(\mathfrak{m}^c))\leq i \reg(Z_1(\mathfrak{m}^c))= i(\reg(H_0(\mathfrak{m}^c))+2)= i(c+1)
$$
Since $i\leq 2c$, then $\reg(Z_i(\mathfrak{m}^c))\leq (i+2)c$. 
\end{proof}

As a consequence of Theorem  \ref{bettiveronese} we know that $\reg(H_i(\mathfrak{m}^c))=(i+1)c+r_i$. From the definition of Green-Lazarsfeld index one can see $i\leq \inde(V_R(c))$ if and only if $r_i \leq c-1$. In order to find a bound for the index of Veronese embedding we can study the behavior of $r_i$'s. Notice that $r_i\in \mathbb{Z}$ for instance $r_0=-1$.

\begin{prop}\label{ri's}
With the above notations we have $r_{i+1}\leq r_{i} +1$  for all $i\leq \min \lbrace \inde (V_S(c))+1, 2c+1 \rbrace$.
\end{prop}
\begin{proof}
By Theorem \ref{convex} we have a triangle inequality between the regularity of Koszul cycles, in particular for $Z_1(\mathfrak{m}^c)$ and $Z_i(\mathfrak{m}^c)$. By using  Corollary \ref{increasehomology} and Theorem \ref{reginduc}
\begin{align*}
\reg (H_{i+1}(\mathfrak{m}^c))+2 &=  \reg(Z_{i+2}(\mathfrak{m}^c))\\
&\leq \reg(Z_{i+1}(\mathfrak{m}^c))+ \reg(Z_1(\mathfrak{m}^c))\\
 &= \reg(H_{i}(\mathfrak{m}^c))+2+ \reg(H_0(\mathfrak{m}^c))+2
\end{align*}

by the above notation we have that $(i+2)c+r_{i+1}\leq (i+1)c+r_{i}+ c+1$. In particular $r_{i+1}\leq r_{i-1}+1$.
\end{proof}

\begin{cor}
The Green-Lazarsfeld index of Veronese embedding $V_S(c)$ is at least $c$.
\end{cor}
\begin{proof}
As we mention above, for finding the Green-Lazarsfeld index of Veronese embedding we should control $r_i$'s.  Proposition \ref{ri's}  shows that in each step, $r_i$'s can be increased only be one. Since $r_0=-1$ so $r_c\leq c-1$ that means $\inde(V_S(c))\geq c$.
\end{proof}

%%%%%%%%%%%%%%%%%%%%%%%%%%%%%%%%%%%%%%%%%%%%%%%%%%%%%%%%%%%%%%%%%%%%%%%%%%

\section*{Acknowledgment}
A part of this work was done during a research visit of first author at Dipartimento di Matematica, Universita di Genova, DIMA. First author would like to thank Aldo Conca for very useful discussions and hospitality of DIMA during his visit. We would like to thank Marc Chardin who provided insight that greatly assisted the research.

%%%%%%%%%%%%%%%%%%%%%%%%%%%%%%%%%%%%%%%%%%%%%%%%%%%%%%%%%%%%%%%%%%%%%%%%%%

Institut de Math\'{e}matiques de Jussieu, UPMC, 75005 Paris, France.\\
\\
kamranlamei79@gmail.com \qquad navid.nemati@imj-prg.fr
%%%%%%%%%%%%%%%%%%%%%%%%%%%%%%%%%%%%%%%%%%%%%%%%%%%%%%%%%%%%%%%%%%%%%%%%%%
\end{document}